\documentclass[microtype]{gtpart}
\usepackage{pinlabel}

\usepackage[latin1]{inputenc}
\usepackage{amsmath, amssymb, amstext, amsfonts, amsbsy, amsthm, fancyhdr, multicol,amscd}
\usepackage{graphicx} 
\usepackage{amsfonts, amscd, amssymb}
\usepackage[all]{xy}
\usepackage[mathscr]{eucal}
\usepackage{amsthm,amsmath,amssymb,amscd} 
\usepackage{amsmath,amsthm,amssymb,graphicx}
\usepackage{epsfig}
\usepackage{enumerate}
\usepackage{multirow}
\usepackage{color}
\usepackage[alphabetic]{amsrefs}
\usepackage{float}

\newcommand{\N}{\ensuremath{\mathbb{N}}}
\newtheorem {theo} {Theorem}

\newtheorem {coro} [theo] {Corollary}

\newtheorem {obs} {Remark}

\def\0{{\bf 0}}

\def\N{\mathbb{N}}

\def\C{\mathbb{C}}

\def \J{J}

\newtheorem{theorem}{Theorem}[section]
\newtheorem{proposition}[theorem]{Proposition}

\newtheorem{lemma}[theorem]{Lemma}

\newtheorem{remark}[theorem]{Remark}

\newtheorem{definition}[theorem]{Definition}

\newtheorem{construction}[theorem]{Construction}

\begin{document}

\title{Geometric Schottky groups and non-compact hyperbolic surfaces with infinite genus}

\author[John A. Arredondo]{John A. Arredondo}
\givenname{John A.}
\surname{Arredondo}
\email{alexander.arredondo@konradlorenz.edu.co}

\address{John A. Arredondo\\\newline Fundaci\'on Universitaria Konrad Lorenz\\\newline CP. 110231, Bogot\'a, Colombia.\\\newline Camilo Ram\'irez Maluendas\\\newline Universidad Nacional de Colombia, Sede Manizales\\\newline Manizales, Colombia.}

\author[Camilo Ram\'irez Maluendas]{Camilo Ram\'irez Maluendas}
\givenname{Camilo}
\surname{Ram\'irez Maluendas}
\email{camramirezma@unal.edu.co}


\begin{abstract}

The topological type of a non-compact Riemann surface is determined by its ends space and the ends having infinite genus. In this paper  for a non-compact Riemann Surface $S_{m,s}$ with $s$ ends and exactly $m$ of them  with infinite genus, such that $m,s\in \N$ and $1<m\leq s$, we give a precise description of the infinite set of generators of a Fuchsian (geometric Schottky) group $\Gamma_{m,s}$ such that the quotient space $\mathbb{H}/ \Gamma_{m, s}$ is homeomorphic to $S_{m,s}$ and has infinite area. For this construction, we exhibit a hyperbolic polygon with an infinite number of sides and give a collection of Mobius transformations identifying the sides in pairs.

\end{abstract}

\maketitle

\textbf{keywords:} \emph{Non-compact surfaces; Hyperbolic surfaces; Schottky groups; Infinitely generated Fuchsian groups.}


\section{Introduction}\label{introduction}

It is well known that for a manifold $M$ and a properly discontinuous action of a group $\Gamma$, the quotient $M / \Gamma$ correspond to a manifold. In particular, every compact Riemann surface of genus $g$ grater than one, is homeomorphic to $\mathbb{H}/ \Gamma_g$, the space of orbits of the hyperbolic plane $\mathbb{H}$ under the action of some finitely generated subgroup $\Gamma_g$ of the automorphism group of $\mathbb{H}$, and the fundamental group of this compact Riemann surface $\mathbb{H}/ \Gamma_g$  is isomorphic to $\Gamma_g$. In this stage, the construction of the Riemann surface $\mathbb{H}/ \Gamma_g$  immediately puts in the map Fuchsian groups, and in particular, classical Schottky groups, in fact, every compact Riemann surface is the quotient of the Riemann sphere by a Schottky group (see \emph{e.g.}, \cite{Bear1}, \cite{Sch}). The previous discussion it is extended by the Uniformization Theorem (see \emph{e.g.} \cite{Abi}, \cite{Far}) to the world of  non-compact Riemann surfaces with infinite genus, in this case the group $\Gamma$ is still Fuchsian, more precisely, we are interested in study a type of them, called \textit{geometric Schottky groups}, which were recently introduced by Anna Zielicz \cite{Ziel}. These groups can be constructed in a similar way to a classical Schottky group but they could be infinitely generated.

Non-compact Riemann surfaces are mainly distinguished by the ends space and the genus associated in each end. Among these surfaces stand out the Infinite Loch Ness monster (up to homeomorphims, it is the surface with only one end and infinite genus) and the Jacob's ladder (up to homeomorphism, it is the only surface with two ends and each one having infinite genus) see \cite{PSul} and \cite{Ghys}, which are some of the usual examples in this field, in fact, in \cite{AyC} the authors give a precisely description of $\Gamma$ a subgroup of $PSL(2,\mathbb{Z})$ to hold the Infinite Loch Ness monster with hyperbolic structure as the quotient $\mathbb{H}/ \Gamma$. Motivated by this particularity, the various investigations on non-compact Riemann surfaces (see \emph{e.g.}, \cite{AVR}, \cite{LKTR}, \cite{MatKat}, \cite{ValRa} among others) and the characterization given by the Uniformization Theorem from the point of view of universal covers for Riemann surfaces, naturally arises the following inquiry: \textit{Given a non compact Riemann Surface $S_{m,s}$ with $s$ ends, and having exactly $m$ of them with infinite genus, such that $m,s\in \N$ and $1<m\leq s$ (see Figure \ref{msends}). Can one give a precise description of the set of generators of a Fuchsian group $\Gamma_{m,s}$, such that the quotient space $\mathbb{H}/ \Gamma_{m, s}$ is homeomorphic to $S_{m,s}$?}

\begin{figure}[h]
\begin{center}
\includegraphics[scale=0.28]{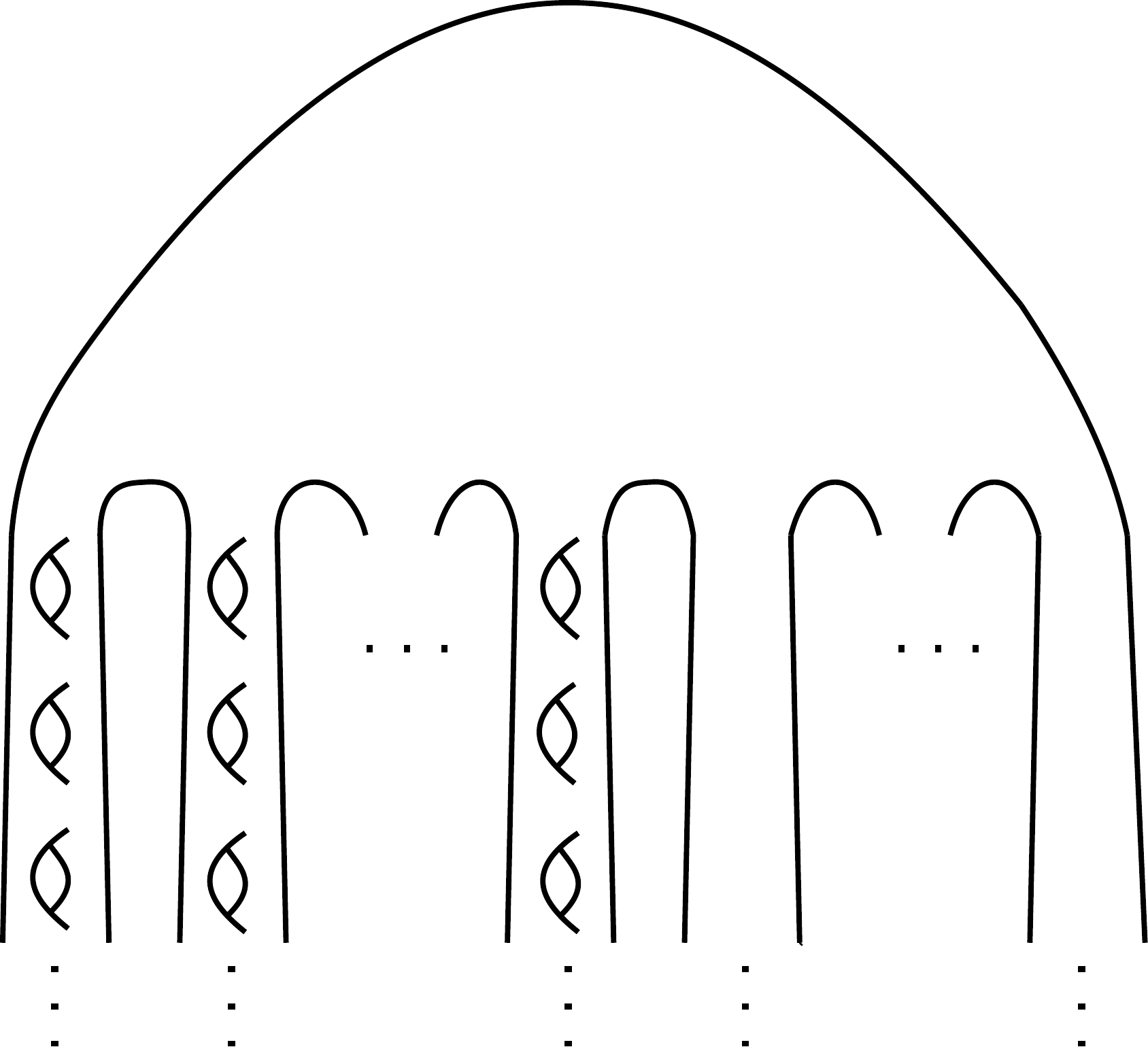}\\
  \caption{\emph{Riemann surface having $s$ ends and $m$ ends with infinite genus.}}
   \label{msends}
\end{center}
\end{figure}

The rest of the paper is  dedicated to answering  this question, we give a precise description of the set of generators of a Fuchsian (geometric Schottky) group   uniformizing a surface of infinite topological type. More precisely we prove the following result:

\begin{theo}\label{T:1}
Let $m$ and $s$ be two positive integers satisfying $1<m\leq s$ and let $\Gamma_{m, s}$ be the geometric Schottky group generated by the M\"{o}bius transformations 
\begin{equation*}
f_t(z)=\dfrac{-5tz+(25t^2-1)}{z-5t}, \quad 
\end{equation*}

\begin{equation*}
\begin{array}{ccc}
g_{k,n}(z)&=&\dfrac{-(17+(5k-3)\cdot 2^n\cdot 10)z+\dfrac{(17+(5k-3)\cdot 2^n\cdot 10)(13+(5k-3)\cdot 2^n\cdot 10)-1}{2^n\cdot 10}}{2^n\cdot 10z-(13+(5k-3)\cdot 2^n\cdot 10)},\\
&&\\
\end{array}
\end{equation*} 

\begin{equation*}
\begin{array}{ccc}
h_{k,n}(z)&=&\dfrac{-(13+(5k-3)\cdot2^{n}\cdot 10)z+\dfrac{(17+(5k-3)\cdot 2^{n}\cdot 10)(13+(5k-3)\cdot2^{n}\cdot 10)-1}{2^n\cdot 10}}{2^n\cdot 10z-(17+(5k-3)\cdot 2^{n}\cdot 10)},\\
&&\\
\end{array}
\end{equation*}
and their inverse, where $t\in \{1,\ldots, s-1\}$, $k\in \{1,\ldots, m\}$ and $n\in \N$. Then the quotient space $S_{m, s}:=\mathbb{H}/\Gamma_{m, s}$ is a geodesic complete Riemann surface having $s$ ends, and exactly $m$ of them with infinite genus.
\end{theo}

It follows from The Uniformization Theorem, see \cite[p.174]{Bear1}.

\begin{coro}
The fundamental group of the Riemann surface $S_{m, s}$ is isomorphic to the geometric Schottky group $\Gamma_{m, s}$.
\end{coro}

The paper is organized as follows: In \textbf{section} \ref{section2} we collect the principal tools used through the paper and \textbf{section} \ref{section3} is dedicated to the proof of our main result, which is divided into five steps: In \textbf{step 1}  we define the group $\Gamma_{s,m}$, introducing $\mathcal{C}$ a suitable family of half circles of the hyperbolic plane, and $J$ the family of M\"{o}bius transformation having as isometric circles the elements of $\mathcal{C}$. Hence,  the elements of $J$ generate the group $\Gamma_{s,m}$.  In \textbf{step 2} we prove that $\Gamma_{s,m}$ is a Fuchsian group with  a  geometric Schottky structure.  In \textbf{step 3} we  prove that $\mathbb{H} / \Gamma_{s,m}$ is a complete geodesic Riemann surface. Finally in  In \textbf{step 4} and   \textbf{step 5} we characterize the ends space and the genus of each end for the Riemman surface.


\section{PRELIMINARIES}\label{section2}

We dedicate this section to introduce some general results of surfaces that will be needed for the proof of the main theorem. Our text is not self-contained, hence we refer the reader to the references  for details.

\subsection{Ends spaces}\label{section2.1}

By a surface $S$ we mean a connected and orientable $2$-dimensional manifold. 

\begin{definition}\label{d:2.1}
\cite{Fre}. Let $U_1\supset U_2\supset \cdots$ be an infinite nested sequence of non-empty connected open subsets of $S$, such that: 

$\ast$ The boundary of $U_n$ in $S$ is compact for every $n\in\mathbb{N}$.

$\ast$ For any compact subset $K$ of $S$ there is $l\in\mathbb{N}$ such that $U_{l}\cap K=\emptyset$. We shall denote the sequence $U_1\supset U_2\supset \cdots$ as $(U_n)_{n\in\mathbb{N}}$. 

Two sequences $(U_{n})_{n\in\mathbb{N}}$ and $(U_{n}^{'})_{n\in \mathbb{N}}$ are equivalent if for any $l \in \mathbb{N}$ it exists $k \in \mathbb{N}$ such that $U_{l}\supset U_k^{'}$ and $n \in \mathbb{N}$ it exists $m \in \mathbb{N}$ such that $U_{n}'\supset U_m$. We will denote the set of ends  by $Ends(S)$ and each equivalence class $[U_{n}]_{n\in\mathbb{N}}\in Ends(S)$ is called an \textbf{end} of $S$. For every non-empty open subset $U$ of $S$ in which its boundary $\partial U$ is compact, we define the set $
U^{*}:=\{[U_{n}]_{n\in\mathbb{N}}\in  Ends(S)\, : \, U_{j}\subset U \text{ for some }j\in\mathbb{N}\}.
$
The collection of all sets of the form $U^{*}$, with $U$ open with compact boundary of $S$, forms a base for the topology of $Ends(S)$.
\end{definition}

For some surfaces their ends space  carries extra information, namely, those ends that carry \textbf{infinite} genus. This data, together with the space of ends and the orientability class, determines the topology of $S$. The details of this fact are discussed in the following paragraphs. Given that, this article only deals with orientable surfaces; from now on, we dismiss the non-orientable case. 

A surface is said to be \textbf{planar} if all of its compact subsurfaces are of genus zero. An end $[U_n]_{n\in\mathbb{N}}$ is called \textbf{planar} if there is $l\in\mathbb{N}$ such that $U_l$ is planar. The \textbf{genus} of a surface $S$ is the maximum of the genera of its compact subsurfaces. Remark that, if a surface $S$ has \textbf{infinite genus}, there is no finite set $\rm \mathcal{C}$ of mutually non-intersecting simple closed curves with the property that $S\setminus \mathcal{C}$ is  \textbf{connected and planar}. We define $Ends_{\infty}(S)\subset Ends(S)$ as the set of all ends of $S$ which are not planar or with infinite genus. It comes from the definition that $Ends_{\infty}(S)$ forms a closed subspace of $Ends(S)$.

\begin{theorem}[Classification of non-compact and orientable surfaces, \cite{Ker}, \cite{Ian}]\label{t:2.2}
Two non-compact and orientable surfaces $S$ and $S'$  having the same genus are homeomorphic if and only if there is a homeomorphism $f: Ends(S)\to Ends(S^{'})$ such that $f( Ends_{\infty}(S))= Ends_{\infty}(S')$.
\end{theorem}

\begin{proposition}\label{p:2.3}
\cite[Proposition 3]{Ian}. The space of ends of a connected surface $S$ is totally disconnected, compact, and Hausdorff. In particular, $Ends(S)$ is homeomorphic to a closed subspace of the Cantor set.
\end{proposition}

\begin{remark}\label{r:2.4}
\cite{SPE} A surface $S$ has exactly $n$ ends if and only if for all compact subset $K \subset S$ there is a compact $K^{'}\subset S$ such as $K\subset K^{'}$ and $S\setminus  K^{'}$ has $n$ connected component.
\end{remark}


\subsection{Isometric circles of the Hyperbolic plane}\label{section2.2}
 
Given the M\"{o}bius transformation $f\in PSL(2,\mathbb{R})$  with $c\neq 0$, the half-circle
$
C(f):=\{z\in\mathbb{H}:| cz + d |^{-2}=1\}
$
will be called the \textbf{isometric circle of} $f$ (see \emph{e.g.}, \cite[p. 9]{MB}). We note that $\dfrac{-d}{c}\in\mathbb{R}$ the center of $C(f)$ is mapped by $f$ onto the infinity point  $\infty$. Further, $f$ sends the half-circle $C(f)$ onto $C(f^{-1})$ the isometric circle of the M\"{o}bius transformation $f^{-1}$,  such that
$
\label{eq:4}
C(f^{-1})=\{z\in\mathbb{H}:| -cz + a |^{-2}=1\}.
$
The \emph{inside} and \emph{outside}, respectively, of the circle $C(f)$ are the sets $\check{C}(f):=\{z\in\mathbb{H}: | cz + d |^{-2}<1\}$ and $\hat{C}(f):=\{z\in\mathbb{H}: | cz + d |^{-2}>1\}$, respectively \cite{MB}.

\begin{remark}
\label{r:2.5}
The isometric circles $C(f)$ and $C(f^{-1})$ have the same radius $r= | c |^{-1}$, and their respective centers are $\alpha=\dfrac{-d}{c}$ and $\alpha^{-1}=\dfrac{a}{c}$. 
\end{remark}

If $C$ is a half-circle with center $\alpha\in\mathbb{R}$ and radius $r>0$, the \emph{M\"{o}bius transformation} $f_C\in PSL(2,\mathbb{R})$ is said be $\emph{inversion in}$ $C$, if it has the hyperbolic geodesic $C$ as fixed points exchanging its the ends points, and maps $\check{C}$ the inside of $C$ (respectively, $\hat{C}$ the outside of $C$) onto $\hat{C}$ the outside of $C$ (respectively, $\check{C}$ the inside of $C$).

\begin{remark}
 \label{r:2.10}
We let $L_{\alpha-2r}$ and $L_{\alpha+2r}$ be the two orthogonal straight lines  to the real axis $\mathbb{R}$ through the points $\alpha-2r$ and $\alpha+2r$, respectively. Then the inversion $f_C$ sends $L_{\alpha-2r}$ (analogously, $L_{\alpha+2r}$) onto the half-circle whose ends points are $\alpha+\frac{r}{2}$ and $\alpha$ (respectively, $\alpha-\frac{r}{2}$ and $\alpha$). See the Figure \ref{Figure5}. Given that $f_C$ is an element of $PSL(2,\mathbb{R})$, then for every $\epsilon < \frac{r}{2}$ the closed hyperbolic $\epsilon$-neighborhood of the half-circle $C$ does not intersect any of the hyperbolic geodesics $L_{\alpha-2r}$, $f_{C}(L_{\alpha-2r})$, $L_{\alpha+2r}$, and $f_{C}(L_{\alpha+2r})$.\begin{figure}[h!]
\begin{center}
\includegraphics[scale=0.5]{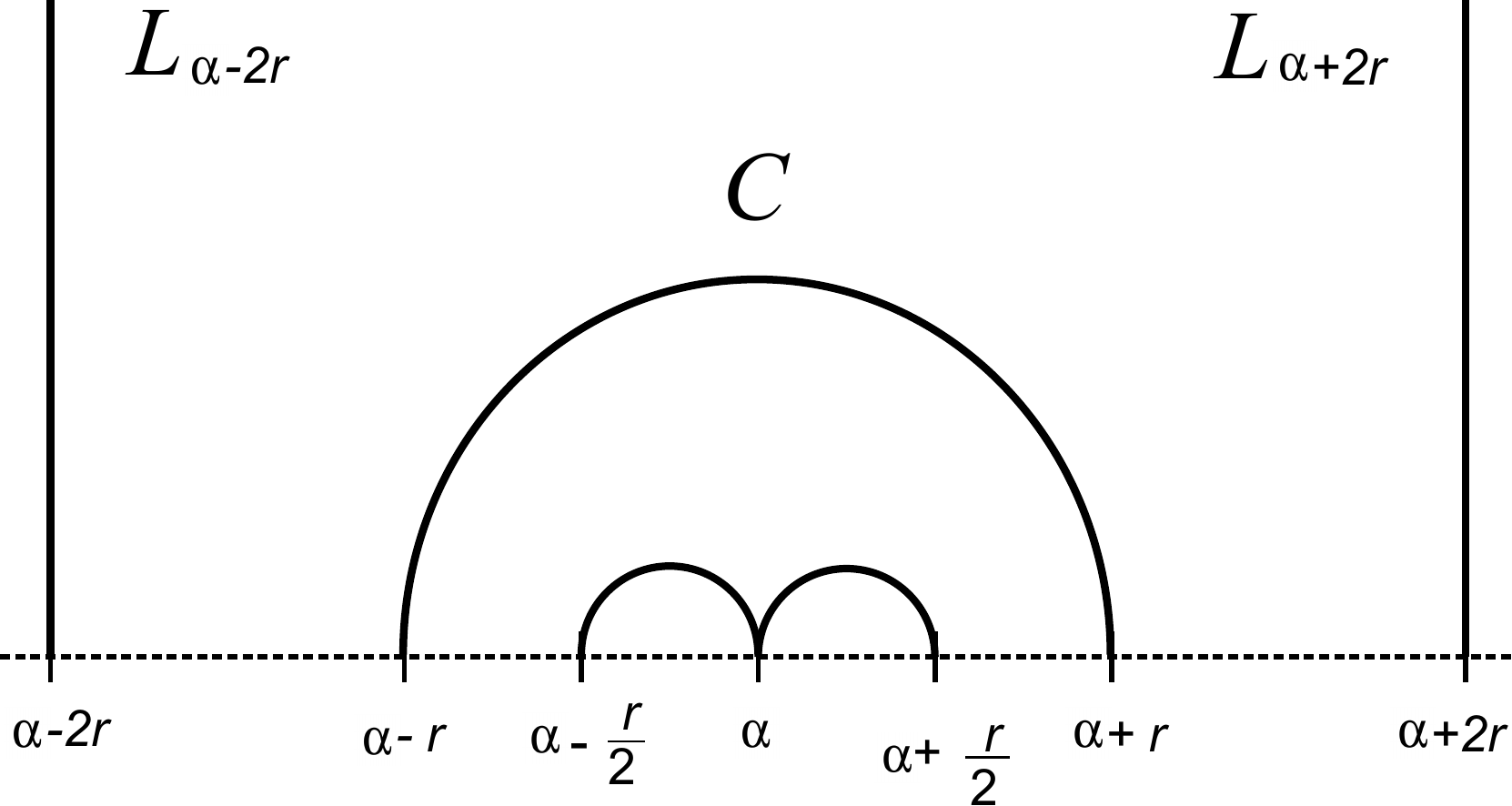}\\
  \caption{\emph{Inversion in $C$.}}
   \label{Figure5}
\end{center}
\end{figure}

\end{remark}

\begin{lemma}\label{l:2.6}
 Let  $C_1$ and $C_2$ be two disjoint half-circles having centers and radius $\alpha_1, \alpha_2 \in\mathbb{R}$, and $r_1, r_2 > 0$, respectively. Suppose  that $|\alpha_1 -\alpha_2 |> (r_1 + r_2)$, then for every $\epsilon<\frac{\min\{r_1, r_2\}}{2}$ the closed hyerbolic $\epsilon$-neighborhoods of the half-circles $C_1$ and $C_{2}$ are disjoint.
\end{lemma}

\begin{proof} By hypothesis $|\alpha_1 -\alpha_2 |> (r_1 + r_2)$, then the open strips $S_1:=\{z\in\mathbb{H}: \alpha_1-2r_1 < Re(z)<\alpha_1+2r_1 \}$ and $S_2:=\{z\in\mathbb{H}: \alpha_{2}-2r_{2} < Re(z)<\alpha_{2}+2r_{2}\}$ are disjoint. We remark that the half-circle $C_i$ belongs to the open strip $S_i$, for every $i\in\{1, 2\}$. Further, the reflection $f_C$,
\begin{equation}\label{eq:9}
 z \mapsto \dfrac{\frac{r_2z}{\sqrt{r_1 r_2}} +\frac{(r_1\alpha_2-r_2\alpha_1)}{\sqrt{r_1 r_2}}}{\frac{r_1}{\sqrt{r_1 r_2}}}
\end{equation}
sends the open strip $S_1$ onto the open strip $S_2$ and vice versa. We denote $r=\min\{r_1, r_2\}$. For each $\epsilon<\frac{r}{2}$, let $\overline{U}_{\epsilon}^i$ be the $\epsilon$-neighborhood of the half-circle $C_i$, with $i\in\{1,2\}$, by Remark \ref{r:2.10} $\overline{U}_{\epsilon}^i\subset S_i$. Given that $f_C$ is an element of $PSL(2,\mathbb{R})$, then $f_{C}(\overline{U}_{\epsilon}^1)=\overline{U}_{\epsilon}^2$. Hence, for all $\epsilon<\frac{r}{2}$, $\overline{U}_{\epsilon}^1 \cap \overline{U}_{\epsilon}^2=\emptyset$.
\end{proof}


\subsection{Fuchsian groups and Fundamental region}\label{section2.4}

\begin{definition}\label{d:2.7}
\cite{KS} Given a Fuchsian group $\Gamma < PSL(2,\mathbb{R})$. A closed region $R$ of the hyperbolic plane $\mathbb{H}$ is said to be \emph{a fundamental region} for $\Gamma$ if it satisfies the following facts:

\textbf{(i)} The union $\bigcup\limits_{f\in \Gamma} f(R)=\mathbb{H}$.

\textbf{(ii)} The intersection of the interior sets $Int(R) \cap f(Int (R)) = \emptyset$ for each $f\in \Gamma \setminus \{Id\}$.

\noindent The difference set $\partial R= R \setminus Int(R)$ is called the \emph{boundary} of $R$ and the family $\mathfrak{T}:=\{f(R): f\in \Gamma\}$ is called the \emph{tessellation} of $\mathbb{H}$.
\end{definition}

 If $\Gamma$ is a Fuchsian group and each one of its elements satisfy that $c\neq 0$, then the subset $R_0$ of $\mathbb{H}$ defined as follows
\begin{equation}
\label{eq:5}
R_0 :=\bigcap\limits_{f\in \Gamma} \overline{\hat{C}(f)}\subseteq \mathbb{H},
\end{equation}
is a fundamental domain for the group $\Gamma$ (see \emph{e.g.}, \cite{Ford}, \cite[Theorem H.3 p. 32]{MB},  \cite[Theorem 3.3.5]{KS}). The fundamental domain $R_0$ is well-known as the \emph{Ford region for} $\Gamma$. 

On the other hand, we can get a Riemann surface from any Fuchsian group $\Gamma$. It is only necessary to define the action $\alpha : \Gamma \times \mathbb{H} \to  \mathbb{H}$ as follows $(f,z) \mapsto f(z)$, which is proper and discontinuous (see \cite[Theorem 8. 6]{KS2}). Now, we define the subset $W:=\{w\in\mathbb{H}: f(w)=w \text{ for any } f\in \Gamma-\{Id\}\}\subseteq\mathbb{H}$. We note that the subset $W$ is countable and discrete, and the action $\alpha$ leaves invariant the subset $W$. Then the action $\alpha$ restricted to the hyperbolic plane $\mathbb{H}$ removing the subset $W$ is free, proper and discontinuous. 
Therefore, the quotient space $S:= (\mathbb{H} \setminus W)/\Gamma$ is well-defined and via the projection map $\pi : (\mathbb{H} \setminus W) \to  S$ such as $z \mapsto [z]$, it comes with a hyperbolic structure, it means, $S$ is a Riemann surface (see \emph{e.g.}, \cite{LJ}).

\begin{remark}\label{r:2.8}
If $R$ is a locally finite\footnote{In the sense due to Beardon  on \cite[Definition 9.2.3]{Bear}.} fundamental domain for the Fuchsian group $\Gamma$, then the quotient space $(\mathbb{H} \setminus W)/\Gamma$ is homeomorphic to $R/\Gamma$ (see Theorem 9.2.4 on \cite{Bear}).
\end{remark}




\subsection{Classical Schottky groups and Geometric Schottky groups}\label{section2.5}

In general, by a classical Schottky group it is understood as a finitely generated subgroup of $PSL(2, \C)$, generated by isometries sending the interior of one circle into the exterior of a different circle, both of them disjoint. From the various definitions, we consider the one given in \cite{MB}, but there are alternative and similar definitions that can be found in \cite{BJ}, \cite{Carne}, \cite{TM}.

Let $C_1, C_1^{\prime}, \ldots, C_n, C_n^{\prime}$ be a set of disjoint countable circles in the extended complex plane $\widehat{\mathbb{C}}$, for any $n\in\mathbb{N}$, bounding a common region $D$. For every $j\in\{1,\ldots,n\}$, we consider the  M\"{o}bius transformation $f_j$, which sends the circle $C_j$ onto the circle $C^{'}_j$, \emph{i.e.}, $f_j(C_j)= C_j^{\prime}$ and $f_j(D)\cap D=\emptyset$. The group $\Gamma$ generated by the set $\{f_j, f_j^{-1}: j\in \{1,\ldots, n\}\}$ is called a \emph{classical Schottky group}.

The \emph{Geometric Schottky groups} can be acknowledged as a nice generalization of the \emph{Classical Schottky groups} because the definition of the first group is extended to the second one, in the sense that the Classical Schottky groups are  finitely generated by definition, while the Geometric Schottkky group can be infinitely generated. Geometric Schottky groups, defined by \emph{Anna Zielicz} (see \cite[Section 3]{Ziel}) are a nice generalization of classical Schottky groups to infinitely many generators, and these are central to our main Theorem \ref{T:1}.

Consider an open interval $A$ in the  real line $\mathbb{R}$. Let $\partial A$ denote its boundary in $\mathbb{R}$, which consists of two points. The two points of $\partial A$ determine a geodesic $\gamma_A$ in $\mathbb{H}$. This geodesic in turn
defines two closed half-planes in $\mathbb{H}$. The boundary at infinity of one of these two closed
half-planes is equal to the closure $\overline{A} := A\cup\partial A$ and we will denote this closed half-plane
by $\hat{A}$. A demostration of $A$, $\gamma_A$ and $\hat{A}$ has been depited in Figure \ref{dibujo}.

\begin{figure}[h]
\begin{center}
\includegraphics[scale=0.6]{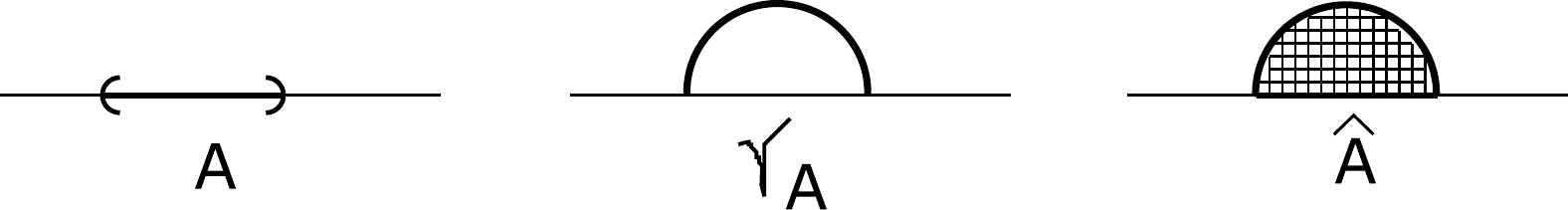}\\
  \caption{\emph{Objects $A$, $\gamma_A$ and $\hat{A}$.}}
   \label{dibujo}
\end{center}
\end{figure}

\begin{definition}\label{d:2.9}
\cite[Definition 2. p. 28]{Ziel} Let $\{A_k:k\in I\}$ be a family of open intervals  with finite length in the real line $\mathbb{R}$, where $I$ is a symmetric subset\footnote{A subset $I\subseteq \mathbb{Z}$ is called \emph{symmetric} if it satisfies that $0\notin I$ and for every $k\in I$ implies $-k\in I$.} of $\mathbb{Z}$ and let $\{f_k:k\in I\}$ be a subset of $PSL(2, \mathbb{R})$. The pair
\begin{equation}\label{eq:10}
\mathfrak{U}(A_k,f_k,I):=(\{A_k\}, \{f_k\})_{k\in I}
\end{equation}
 is called a \emph{Schottky description}\footnote{The writer gives this definition to the Poincar\'e disc and we use its equivalent to the half plane.} if it satisfies the following conditions:
\begin{enumerate}
  \item The closures subsets $\overline{A_k}$ in $\mathbb{C}$ are mutually disjoint.
  \item None of the $\overline{A_k}$ contains a closed half-circle.
  \item For every $k\in I$, we denote as $C_k$ the half-circle whose ends points coincide to the ends points of $\overline{A_k}$, which  is the isometric circle of $f_k$. Analogously, the half-circle $C_{-k}$ is the isometric circle of $f_{-k}:=f^{-1}_k$.
  \item For each $k\in I$, the M\"obius transformation $f_k$ is hyperbolic.
  \item There is an $\epsilon> 0$ such that the closed hyperbolic $\epsilon$-neighborhood of the half-circles $C_{k}$, $k\in I$ are pairwise disjoint.
\end{enumerate}
\end{definition}

\begin{definition}\label{d:2.10}
\cite[Definition 3. p. 29]{Ziel} A subgroup $\Gamma$ of $PSL(2,\mathbb{R})$ is called \emph{Schottky type} if there exists a Schottky description $\mathfrak{U}(A_k,f_k,I)$ such that the generated group by the set $\{f_k:k\in I\}$ is equal to $\Gamma$, \emph{i.e.}, we have 
\begin{equation}\label{eq:generated}
\Gamma=\langle f_k : k\in I \rangle.
\end{equation}
\end{definition}

\begin{definition}\label{Geometric_schokkty_group}
A subgroup $\Gamma$ of $PSL(2,\mathbb{R})$ together a Schottky description $\mathfrak{U}(A_k,f_k,I)$ satisfying property \ref{eq:generated} will be referred to as \emph{geometric Schottky group}.  
\end{definition}

We note that any Schottky description $\mathfrak{U}(A_k,f_k,I)$ defines a Geometric Schottky group.


\begin{proposition}\label{p:2.11}
\cite[Proposition 4. p.34]{Ziel} Every Geometric Schottky group $\Gamma$ is a Fucshian group.
\end{proposition}

The \emph{standard fundamental domain} for the Geometric Schottky group $\Gamma$ having Schottky description  $\mathfrak{U}(A_k,f_k, I)$ is the intersection of all outside of the half-circle associated to the transformations  $f_k$, \emph{i.e.},
\begin{equation}\label{eq:11}
F(\Gamma):=\bigcap_{k\in I} \overline{\hat{C}_{k}}\subset \mathbb{H}.
\end{equation}

\begin{proposition}\label{p:2:12}
\cite[Proposition 2. p. 33]{Ziel} The standard fundamental domain $F(\Gamma)$ is a fundamental domain for the Geometric Schottky group $\Gamma$.
\end{proposition}


\section{Proof of main result} \label{section3}

\textbf{Step 1. Building the group $\Gamma$.} Given that we shall construct a Riemann Surface with $s$ ends for $s\in \mathbb{N}$,  then for each $t\in\{1,\ldots,s-1\}$ we denote as  $C(f_t):=\{z\in\mathbb{H}: |z-5t|=1\}$ and $C(f_t^{-1}):=\{z\in\mathbb{H}: |z+5t|=1\}$  the half-circle having center $5t$ and $-5t$ in the real line, respectively, and the same radius equality to $1$ (see Figure \ref{step_1_1}). Let $f_t$ and $f^{-1}_t$ be the M\"{o}bius transformations having as isometric circle the half-circle $C(f_t)$ and $C(f_t^{-1})$ respectively. Using Remark \ref{r:2.5} we compute and get  the values of the coefficients of $f_t$ and $f_t^{-1}$, it means 
\begin{equation}\label{eq:12}
f_t(z)=\dfrac{-5tz+(25t^2-1)}{z-5t}, \quad f_t^{-1}(z)=\dfrac{-5tz-(25t^2-1)}{-z-5t}.
\end{equation}

\begin{figure}[h]
\begin{center}
\includegraphics[scale=0.58]{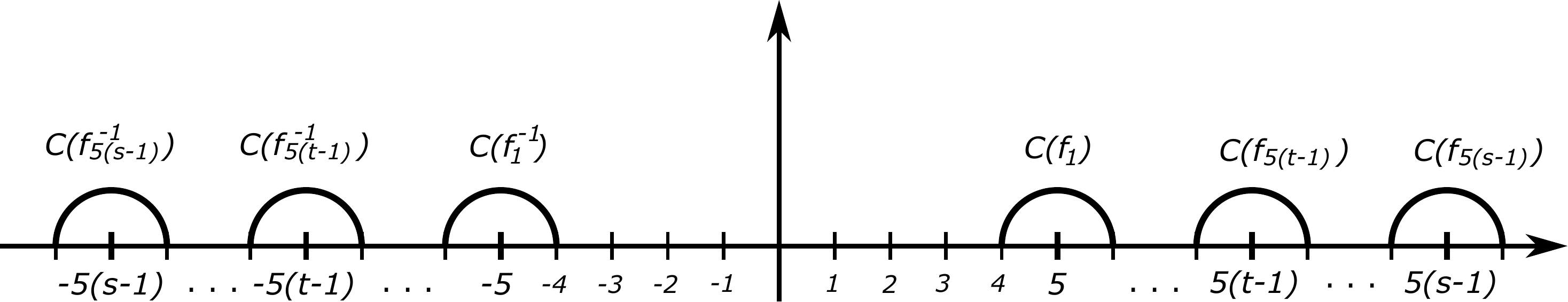}\\
  \caption{\emph{Half-circles $C(f_t)$ and $C(f_t^{-1})$ for each $t\in\{1,\ldots, s-1\}$.}}
   \label{step_1_1}
\end{center}
\end{figure}

\begin{remark}
For each $t\in\{1,\ldots,s-1\}$, we note that the M\"{o}bius transformations $f_t$ and  $f_t^{-1}$  are both hyperbolic.
\end{remark}

In order to hold the Riemann surface having $m\leqslant s$ ends with infinite genus, it is necessary to introduce the following infinite families of half-circles and M\"{o}bius transformations. 

For each $k\in\{1,\ldots,m\}$ we consider the closed intervals $I_k:=[5k-3,5k-2]$ and $\hat{I}_k:=[-(5k-2),-(5k-3)]$ on the real line, which can be written as the following union of closed subintervals
\begin{equation}\label{eq:14}
\begin{array}{ccl}
I_k & = & \bigcup\limits_{n\in\mathbb{N}} \left[5k-3+\dfrac{1}{2^{n}},5k-2+\dfrac{1}{2^{n-1}}\right],\\
\hat{I}_k & = &\bigcup\limits_{n\in\mathbb{N}} \left[-\left(5k-3+\dfrac{1}{2^{n}}\right), -\left(5k-2+\dfrac{1}{2^{n-1}}\right)\right]. 
\end{array}
\end{equation}

We shall introduce the following notation for the closed subintervals in equation (\ref{eq:14}) (see Figure \ref{step_1_2}).
\begin{equation}\label{eq:15}
\begin{array}{ccl}
I_{k,n}&:=&\left[5k-3+\dfrac{1}{2^{n}},5k-2+\dfrac{1}{2^{n-1}}\right],\\
\hat{I}_{k,n}&:=&\left[-\left(5k-3+\dfrac{1}{2^{n}}\right), -\left(5k-2+\dfrac{1}{2^{n-1}}\right)\right].
\end{array}
\end{equation}

\begin{figure}[h]
\begin{center}
\includegraphics[scale=0.6]{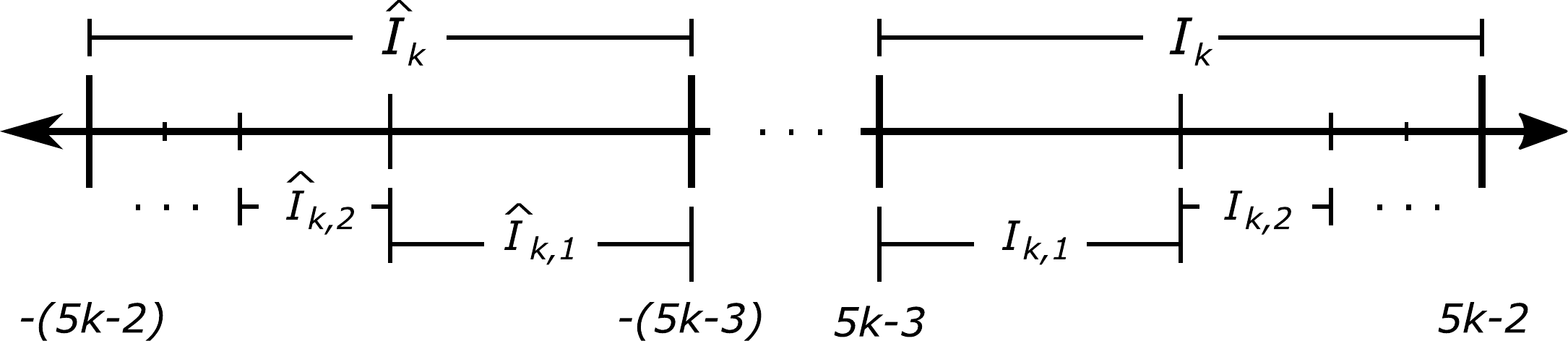}\\
  \caption{\emph{Closed intervals $I_k$ and $\hat{I}_k$ as the union of closed subintervals.}}
   \label{step_1_2}
\end{center}
\end{figure}

\begin{remark}
By construction, the intervals  $I_{k,n}$ and $\hat{I}_{k,n}$ have length $\frac{1}{2^n}$, for each $n\in\mathbb{N}$. Moreover, they are symmetric
 with respect to the imaginary axis. 
\end{remark}

Now, we consider the following two points in $I_{k,n}$
\begin{equation}\label{eq:16}
\begin{array}{cclcl}
\alpha_{k,n} & := & 5k-3+\dfrac{1}{2^n}+\dfrac{3}{2^n\cdot 10}& = &\dfrac{13+(5k-3)\cdot 2^n\cdot 10}{2^n\cdot 10} ,\\
  &&\\
\beta_{k,n} & := & 5k-3+\dfrac{1}{2^n}+\dfrac{7}{2^n\cdot 10}& = & \dfrac{17+(5k-3)\cdot 2^n\cdot 10}{2^n\cdot 10}. \\    
\end{array}
\end{equation}
Analogously, we consider the following two points  in $\hat{I}_{k,n}$
\begin{equation}\label{eq:17}
\begin{array}{ccc}
\alpha_{k,n}^{-1} :=  -\dfrac{17+(5k-3)\cdot 2^n\cdot 10}{2^n\cdot 10},
&&
\beta_{k,n}^{-1} := -\dfrac{13+(5k-3)\cdot 2^n\cdot 10}{2^n\cdot 10}.   
\end{array}
\end{equation}
Note that $\alpha_{k,n}^{-1}$  is not equal to $-\alpha_{k,n}$, but to $-\beta_{k,n}$. Then we denote  as   $C(g_{k,n})$, $C(g_{k,n})$ $C(h_{k,n})$, and $C(h_{k,n}^{-1})$ the half-circles with center in $\alpha_{k,n}$, $\alpha_{k,n}^{-1}$, $\beta_{k,n}$, and $\beta_{k,n}^{-1}$ respectively, and radius $r_{n}:=\frac{1}{2^n \cdot 10}$ (see Figures \ref{step_1_3} and \ref{step_1_4}).

\begin{figure}[h]
\begin{center}
\includegraphics[scale=0.6]{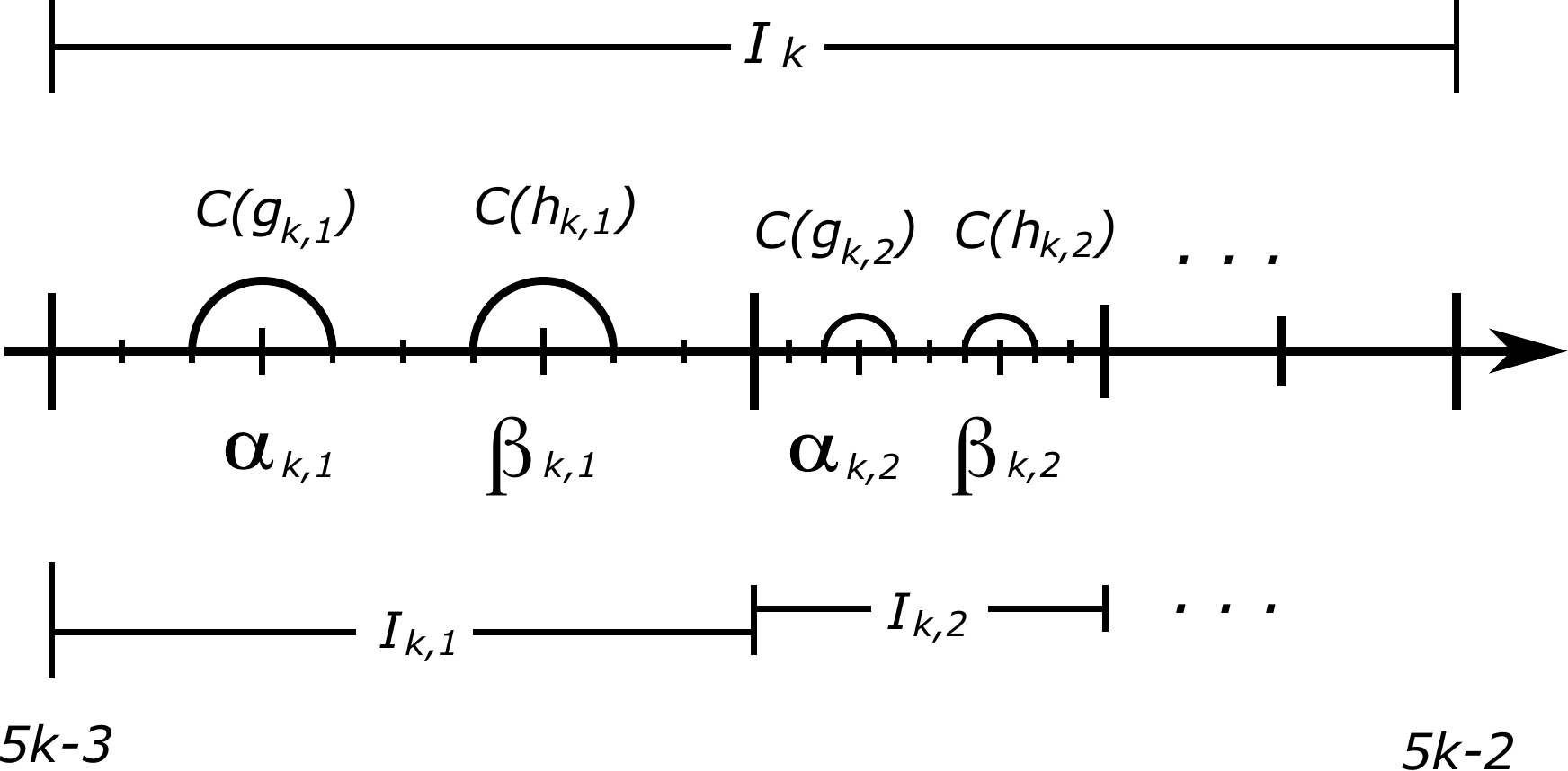}\\
  \caption{\emph{Location of half-circles denoted as $C(g_{k,n})$ and $C(h_{k,n})$.}}
   \label{step_1_3}
\end{center}
\end{figure}

\begin{figure}[h]
\begin{center}
\includegraphics[scale=0.6]{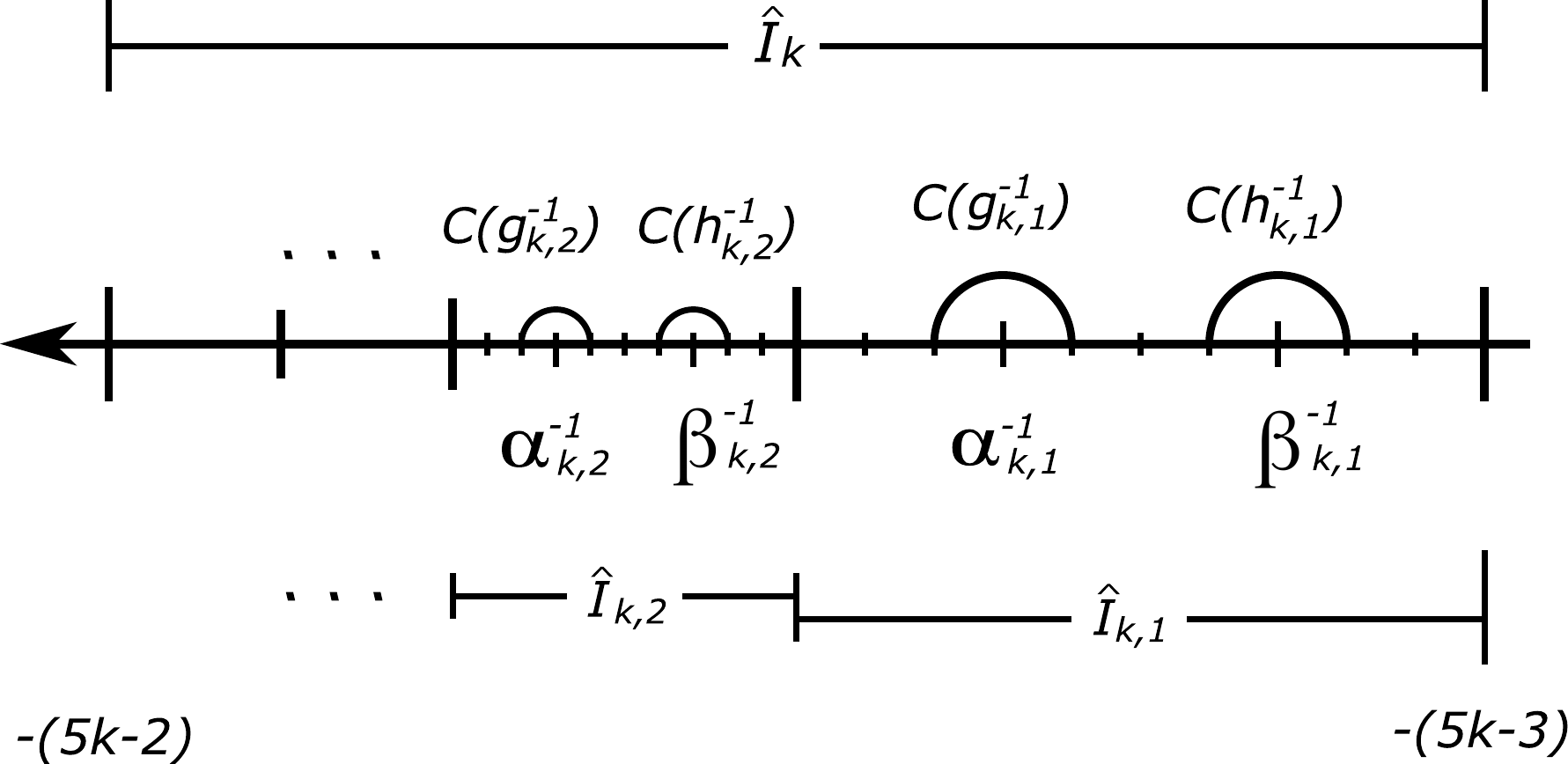}\\
  \caption{\emph{Location of half-circles $C(g^{-1}_{k,n})$ and $C(h^{-1}_{k,n})$.}}
   \label{step_1_4}
\end{center}
\end{figure}

Now, we shall calculate the values of the coefficients of the M\"{o}bius transformation $g_{k,n}(z)=\dfrac{a_{k,n}z+b_{k,n}}{c_{k,n}z+d_{k,n}}$ and its respective inverse
$g_{k,n}^{-1}(z)=\dfrac{d_{k,n}z-b_{k,n}}{-c_{k,n}z+a_{k,n}}$, which have as isometric circles the half-circles $C(g_{k,n})$ and $C(g^{-1}_{k,n})$, respectively. By Remark \ref{r:2.5} we get $c_{k,n} = 2^n\cdot 10$, 
$a_{k,n} =-(17+(5k-3)\cdot 2^n\cdot 10)$ and $d_{k,n}= -(13+(5k-3)\cdot 2^n\cdot 10)$. Now, we substitute these values in the determinant $a_{1,n}d_{1,n}-b_{1,n}c_{1,n}=1$, and computing we hold
\begin{equation*}\label{eq:19a}
b_{k,n}=\dfrac{(17+(5k-3)\cdot 2^n\cdot 10)(13+(5k-3)\cdot 2^n\cdot 10)-1}{2^n\cdot 10}.
\end{equation*}
Hence, we get
\begin{equation}\label{eq:20}
\begin{array}{ccc}
g_{k,n}(z)&=&\dfrac{-(17+(5k-3)\cdot 2^n\cdot 10)z+\dfrac{(17+(5k-3)\cdot 2^n\cdot 10)(13+(5k-3)\cdot 2^n\cdot 10)-1}{2^n\cdot 10}}{2^n\cdot 10z-(13+(5k-3)\cdot 2^n\cdot 10)},\\
&&\\
g^{-1}_{k,n}(z) & =&\dfrac{-(13+(5k-3)\cdot 2^n\cdot 10)z-\dfrac{(17+(5k-3)\cdot 2^n\cdot 10)(13+(5k-3)\cdot 2^n\cdot 10)-1}{2^n\cdot 10}}{-2^n\cdot 10z-(17+(5k-3)\cdot 2^n\cdot 10)}.
\end{array}
\end{equation} 

\begin{remark}\label{r:3.4}
The M\"{o}bius transformations $g_{k,n}$ and $g_{k,n}^{-1}$  are hyperbolic.
\end{remark}

Likewise, we shall calculate the the values of the coefficients of the M\"{o}bius transformation $h_{k,n}(z)=\dfrac{a_{k,n}z+b_{k,n}}{c_{k,n}z+d_{k,n}}$ and its respective inverse
$h_{k,n}^{-1}(z)=\dfrac{d_{k,n}z-b_{k,n}}{-c_{k,n}z+a_{k,n}}$, which have as isometric circles the half-circles $C(h_{k,n})$ and $C(h^{-1}_{k,n})$, respectively. By Remark \ref{r:2.5} we get $c_{k,n} = 2^n\cdot 10$, $ 
a_{k,n}  = -(13+(5k-3)\cdot 2^n\cdot 10)$ and $d_{k,n}=-(17+(5k-3)\cdot2^{n}\cdot 10)$. Now, we substitute these values in the determinant $a_{k,n}d_{k,n}-b_{k,n}c_{k,n}=1$, and computing we hold
$$
b_{k,n}=\dfrac{(17+(5k-3)\cdot 2^{n}\cdot 10)(13+(5k-3)\cdot 2^{n}\cdot 10)-1}{2^n\cdot 10}.
$$
Hence, we get
\begin{equation}\label{eq:23}
\begin{array}{ccc}
h_{k,n}(z)&=&\dfrac{-(13+(5k-3)\cdot2^{n}\cdot 10)z+\dfrac{(17+(5k-3)\cdot 2^{n}\cdot 10)(13+(5k-3)\cdot2^{n}\cdot 10)-1}{2^n\cdot 10}}{2^n\cdot 10z-(17+(5k-3)\cdot 2^{n}\cdot 10)},\\
&&\\
h^{-1}_{k,n}(z)&=&\dfrac{-(17+(5k-3)\cdot 2^{n}\cdot 10)z-\dfrac{(17+(5k-3)\cdot 2^{n}\cdot 10)(13+(5k-3)\cdot 2^{n}\cdot 10)-1}{2^n\cdot 10}}{-2^n\cdot 10z-(13+(5k-3)\cdot 2^{n}\cdot 10)}.\\
\end{array}
\end{equation}

\begin{remark}\label{r:3.5}
The M\"{o}bius transformations $h_{k,n}$ and $h_{k,n}^{-1}$  are hyperbolic.
\end{remark}

From the previous construction of M\"{o}bius transformations and their respectively half-circles (see equations (\ref{eq:12}), (\ref{eq:20}), and (\ref{eq:23})) we define the sets
\begin{equation}\label{eq:24}
\begin{array}{l}
\mathcal{J}:=\{f_t, \,f_t^{-1}, \,g_{k,n}, \,g^{-1}_{k,n}, \,h_{k,n}, \,h^{-1}_{k,n}:t\in\{1,\ldots,s-1\}, k\in\{1,\ldots,m\}, n\in\mathbb{N}\},\\
\\
\mathcal{C} :=\{C(f_t), C(f_t^{-1}), C(g_{k,n}), C(g^{-1}_{k,n}), C(h_{k,n}), C(h^{-1}_{k,n}):t\in\{1,\ldots,s-1\}, k\in\{1,\ldots,m\}, n\in\mathbb{N}\}.
\end{array}
\end{equation}
Then let $\Gamma_{m,s}$ be  the subgroup of $PSL(2,\mathbb{R})$ generated by $\mathcal{J}$. By construction the elements of $\Gamma_{m,s}$ are hyperbolic and the half-circle of $\mathcal{C}$ are pairwise disjoint.

\textbf{Step 2. The group $\Gamma_{m,s}$ is a Fuchsian group.} We will prove that $\Gamma_{m,s}$ is a Geometric Schottky group \emph{i.e.}, we shall define a Schottky description for $\Gamma_{m,s}$. Hence, by Proposition \ref{p:2.11} we will conclude that $\Gamma_{m,s}$ is Fuchsian.

We consider $P=\{p_n\}_{n\in\mathbb{N}}$ the set conformed by all primes numbers. Hence,  the map $\psi : \mathcal{J}\to \mathbb{Z}$ defined by
\[
\begin{array}{rclcrcl}
f_t&\mapsto& p_1^t,& &f_t^{-1}&\mapsto& -p_1^{t},\\
g_{k,n}&\mapsto& p_2\cdot p_{4+n}^k, && g^{-1}_{k,n}&\mapsto& -p_2\cdot p_{4+n}^k,\\
h_{k,n}&\mapsto& p_3\cdot p_{4+n}^k,&& h^{-1}_{k,n}&\mapsto& -p_3\cdot p_{4+n}^k,\\
\end{array}
\]
is injective, for every $t,k,n\in\mathbb{N}$. We remark that $I:=\psi(\mathcal{J})$, the image of $\mathcal{J}$ under $\psi$ is a symmetric subset of the integers numbers, then for each element $k$ of $I$ there is a unique M\"{o}bius transformation $f$ belonging to $\mathcal{J}$ such that $\psi(f)=k$, then we label the transformation $f$ as $f_k$. Therefore, we re-write the sets in the equation (\ref{eq:24}) as
\begin{equation}\label{eq:25}
\begin{array}{ccc}
\mathcal{J}:=\{f_k: k\in I\},&&\
\mathcal{C} :=\{C(f_k): k\in I\}.
\end{array}
\end{equation}
Similarly, the center and radius of each half circle $C(f_k)$ is re-written as $\alpha_k$ and $r_k$, respectively. Now, we define the set $\{A_k:k\in I\}$ being $A_k$ the straight segment in the real line $\mathbb{R}$ having as ends point the same ones of the half circle $C(f_k)$. So, we shall prove that the pair
\begin{equation}\label{eq:26}
\mathcal{U}(A_k,f_k,I):=(\{A_k\},\{f_k\})_{k\in I}
\end{equation} 
is a Schottky description. We note that by the inductive construction of the family $\mathcal{J}$ described above (see equation (\ref{eq:25})), the object $\mathcal{U}(A_k,f_k,I)$ satisfies conditions 1 to 4  in Definition \ref{d:2.9}. Thus, we must only prove that the requirement 5 of the Definition \ref{d:2.9} is done. 

Given the  M\"{o}bius transformation $f_k\in \mathcal{J}$ and its respective half circle $C(f_k)\in\mathcal{C}$ we construct the open vertical strip $M_k:=\{ z\in\mathbb{H}: \alpha_{k}-2r_k < Re(z) < \alpha_{k}+2r_k\}$, being $\alpha_k$ and $r_{k}$ the center and radius of $C(f_k)$, respectively. 

\begin{remark}
We have the following facts.  
\begin{enumerate}
\item The half circle $C(f_k)$ is contained into the open strips $M_k$ \emph{i.e.}, $C(f_k)\subset M_k$.

\item The $\epsilon$-neighborhood $B_{\epsilon}^k$ of the half circle $C(f_k)$ is contained into the open strips $M_k$, choosing $\epsilon<\frac{r_k}{2}$.

\item For any two transformations $f_k\neq f_j\in \mathcal{J}$ the intersection of their open strips associated is empty, $M_k\cap M_j =\emptyset$.
\end{enumerate}
\end{remark}

On the other hand, there is an element $i\in I$ such that the radius $r_i$ of the half circle $C(f_i)\in\mathcal{C}$ is equal to $1$ \emph{i.e.}, $r_i=1$. So, $r_k\leq r_i=1$ for all $k\in I$. Then, by construction we have the relation $| \alpha_k -\alpha_i | > r_k+ r_i=r_k+1, \text{ for all } k\in I$, applying  Lemma \ref{l:2.6} we have that for every $\epsilon<\frac{\min \{r_k,r_i\}}{2}=\frac{r_i}{2}=\frac{1}{2}$ the closure of the $\epsilon$-neighborhoods $B_{\epsilon}^k$ and $B_{\epsilon}^i$ of the half circles $C(f_k)$ and $C(f_i)$ are disjoint. Moreover, this lemma assures that the closure of the $\epsilon$-neighborhood $B_{\epsilon}^k$ is contained in the open strip $M_{k}$. This implies that the closed hyperbolic $\epsilon$-neighborhood $B_{\epsilon}^k$, $k\in I$, are pairwise disjoint, fixing $\epsilon<\frac{1}{2}$. Therefore, requirement 5 of  Definition \ref{d:2.9} is done, then we can conclude that $\mathcal{U}(A_k,f_k,I)$ is a Schottky description. Finally, we consider the set $\mathcal{J}$ described in equation (\ref{eq:25}). Then from Proposition \ref{p:2.11} we hold that the Geometric Schottky group $\Gamma_{m,s}:=\langle f_k :k\in I\rangle $ is a Fuchsian group.

\textbf{Step 3. Holding the suitable surface.} We remember that the Fuchsian group $\Gamma_{m,s}$ defined above acts freely and properly discontinuously on the subset $\mathbb{H} \setminus W$, being $W:=\{w\in\mathbb{H}: f(w)=w \text{ for any } \Gamma_{m,s}-\{Id\}\}$. Nevertheless, in this case the set $W$ is the empty set, because the half circles belonging to $\mathcal{C}$ are pairwise disjoint. Then, the quotient space 
\begin{equation*}
S_{m,s}:=\mathbb{H}/\Gamma_{m,s}
\end{equation*}
is a geodesic complete Riemann surface via the projection map
$\pi:\mathbb{H}\to S_{m,s}$, such as $z\mapsto [z]$  (see \cite{Sch}). 

\textbf{Step 4. The surface $S_{m,s}$ has $s$ ends.} We claim that the surface $S_{m,s}$ has $s$ ends and exactly  $m$ of them with infinite genus. To prove this, it is necessary to introduce the following construction.

\begin{construction}\label{cons:3.7}
 For the hyperbolic plane $\mathbb{H}$ there is an increasing sequence of connected compact sets ${K_1}\subset {K_2}\subset \ldots$ such that $\mathbb{H}=\bigcup\limits_{l\in\mathbb{N}}{K_l}$, and whose complements define the only one end of the hyperbolic plane. In other words, $
\mathbb{H} \setminus {K_1}\supset \mathbb{H}\setminus {K_2}\supset \ldots
$ the infinite nested sequence of non-empty connected  subsets of $\mathbb{H}$, defines the equivalent class $[\mathbb{H} \setminus {K_l}]_{l\in\mathbb{N}}$, which is the only one end of $\mathbb{H}$ (see Definition \ref{d:2.1}). More precisely, for each $l \in\mathbb{N}$, we define the compact subset
\begin{equation}\label{eq:28}
{K_l}:=\left\{z\in\mathbb{H}: -5(s-1)-l\leqslant Re(z)\leqslant 5(s-1)+l \text{ and } \frac{1}{l}\leqslant Im(z)\leqslant l+1 \right\}.
\end{equation}
We remark that if ${K}$ is a compact subset of $\mathbb{H}$, then there is $l\in\mathbb{N}$ such that ${K}\subset {K_l}$ (see Figure \ref{step_1_5}).  
\begin{figure}[h]
\begin{center}
\includegraphics[scale=0.5]{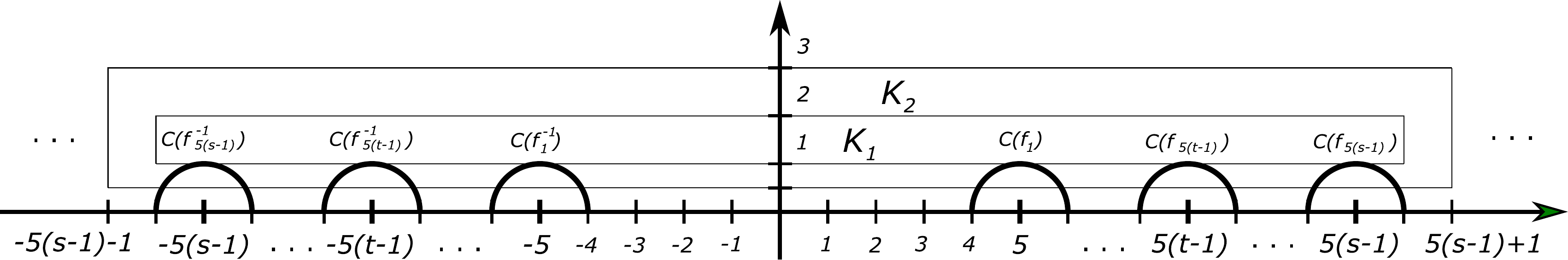}\\
  \caption{\emph{Geometric representation of the compact subset ${K_l}$ on the hyperbolic plane.}}
   \label{step_1_5}
\end{center}
\end{figure}

Now, for each $t\in\{1,\ldots, s\}$ we define the open subset $U_t$ of the hyperbolic plane $\mathbb{H}$ as follows

$\ast$ If  $t\in\{1,\ldots, s-1\}$, then $U_t$ is the union
\begin{equation}
\begin{array}{ccl}
U_{t} & := & \{z\in\mathbb{H}: -5t \leq Re(z)\leq -5(t-1) \text{ and } 0< Im(z)< 1\}\bigcup   \\ 
         &       &\{z\in\mathbb{H}: 5(t-1) \leq Re(z)\leq 5t \text{ and } 0< Im(z)< 1\} .
\end{array}
\end{equation}

$\ast$ Otherwise, if $t=s$, then $U_t=U_s$ is the complement

\begin{equation}
\begin{array}{ccl}
U_t=U_s & :=& \mathbb{H}\setminus \{z\in\mathbb{H}: -5(t-1)\leq Re(z)\leq 5(t-1) \text{ and } 0<Im(z)\leq 1 \}.
\end{array}
\end{equation}

%
%
%
%
We remark that the open subsets $U_1,\ldots,U_s$ are pair disjoint.
\end{construction}

On the other hand, following  Proposition \ref{p:2:12} the standard fundamental domain
\begin{equation}
F(\Gamma_{m,s})=\bigcap_{k\in I} \overline{\hat{C}(f_k)}\subset \mathbb{H},
\end{equation}
is a fundamental domain for the Fuchsian group $\Gamma_{m,s}$. We note that the fundamental domain $F(\Gamma_{m,s})$ is connected, locally connected and its boundary is the set of half circles $\mathcal{C}$  (see equation (\ref{eq:24})). It is easy to check that the hyperbolic area of $F(\Gamma_{m,s})$ is infinite.

By construction, if $\gamma$ is a horizontal line belonged to the horizontal strip 
$
M:=\{z\in\mathbb{H}:0< Im(z) < 1\},
$
then the image of $\gamma\cap F(\Gamma_{m,s})$ under $\pi$ are $s$ disjoint curves. We shall abuse of the language and the word horocycle would mean horizontal line. Contrary, if $\gamma$ is a horocycle belonged to $\mathbb{H}\setminus M$, then the image of $\gamma\cap F(\Gamma_{m,s})=\gamma$ under the projection $\pi$ is a curve. Considering the objects described in  Construction \ref{cons:3.7}, it implies that for each $t\in\{1,\ldots,s\}$ and each $l\in\mathbb{N}$ the image of $(U_t \setminus {K_l})\cap F(\Gamma_{m,s})$ under $\pi$ is a path-wise connected subset of $S_{m,s}$, then the open subset 
\begin{equation}
C_{t,l}:=\pi((U_t \setminus {K_l})\cap F(\Gamma_{m,s}))\subset S_{m,s}
\end{equation} 
is connected and its boundary  in $S_{m,s}$ is compact for each $l\in\mathbb{N}$ and $t\in\{1,\ldots,s\}$. By definition of ${K_l}$ we hold  that 
$$
S_{m,s}\setminus \pi(K_l)=S_{m,s} \setminus \pi(K_{l}\cap F(\Gamma_{m,s})) =\bigcup\limits_{t=1}^{s} C_{t,l}.
$$
Let $\hat{K}$ be a compact subset of $S_{m,s}$, by Remark \ref{r:2.4} we must prove that there exist a compact subset $K^{'}$ of $S_{m,s}$ such that $\hat{K} \subset K^{'}$ and  $S_{m,s} \setminus K^{'}$ are exactly $s$ connected component. Let ${K}$ denote the compact subset of $\mathbb{H}$ such that $\pi({K})=\hat{K}$. From Construction \ref{cons:3.7} there is a compact subset ${K_l}$ of the hyperbolic plane $\mathbb{H}$ such that ${K}\subset {K_l}$, for any $l\in\mathbb{N}$.  Let $K^{'}$ be the image of ${K_{l}}\cap F(\Gamma_{m,s})$ under the projection $\pi$, $K^{'}:=\pi({K_{l}}\cap F(\Gamma_{m,s}))$ . Then  we shall prove that $S_{m,s} \setminus K^{'}$ are exactly $s$ connected components, it means that $S_{m,s} \setminus K^{'}$  is the disjoint union 
\[
S_{m,s} \setminus K^{'}=\bigsqcup\limits_{t=1}^{s} C_{t,l}.
\] 
We just should  prove that the subsets $C_{t,l}$ are disjoint. We will proceed by contradiction and will suppose that there is $u\neq t\in\{1,\ldots,s\}$ such that $C_{u,l}\cap C_{t,l}\neq \emptyset$. Let $z$ be an element in the fundamental domain $F(\Gamma_{m,s})$ such that $\pi(z)\in C_{u,l}\cap C_{t,l} $, then $z\in (U_u \setminus {K_l})\cap F(\Gamma_{m,s})$ and $z\in (U_t \setminus {K_l})\cap F(\Gamma_{m,s})$, which implies that $z$ is an element on the intersection of $U_{u}\cap U_{t}$. Clearly, this is a contradiction because  the open subsets $U_{s}$ are disjoint (see Construction \ref{cons:3.7}).  This proves that the surface $S_{m,s}$ has $s$ ends.

\textbf{Step 5. The surface $S_{m,s}$ has $m$ ends with infinite genus.} Given the exhaustion of $\mathbb{H}=\bigcup\limits_{l\in\mathbb{N}} {K_l}$ by compact subsets in  Construction \ref{cons:3.7}, the image of each element of the family $\{{K_l}\cap F(\Gamma_{m,s})\}_{l\in\mathbb{N}}$ under the projection $\pi$ is  a compact subset of $S_{m,s}$, which we denote as $\mathcal{K}_l$. Hence, the family $\{\mathcal{K}_l\}_{l\in\mathbb{N}}$  is also  an exhaustion of $S_{m,s}=\bigcup\limits_{l\in\mathbb{N}} \mathcal{K}_l$ by compact subsets. Using the ideas above, we can write the set 
\[
S_{m,s}\setminus \mathcal{K}_l=\bigsqcup_{t=1}^s C_{t,l},
\]
where each $C_{t,l}$ with $t\in\{1,\ldots,s\}$ and $l\in\mathbb{N}$ is a connected component whose boundary is compact in $S_{m,s}$ and, $C_{t,l}\supset C_{t,l+1}$ . In other words, the ends space $Ends(S_{m,s})$ are all the nested sequences $(C_{t,l})_{l\in\mathbb{N}}$ \emph{i.e.},  $Ends(S_{m,s})=\{[C_{t,l}]_{l\in\mathbb{N}}: t\in\{1,\ldots,s\}\}$. By construction follows that the ends $[C_{t,l}]_{l\in\mathbb{N}}$ whit $t\in\{m+1,\ldots,s\}$ is planar because of the subsurface $C_{t,l}$ is homeomorphic to the cylinder for all $l\in\mathbb{N}$ and all $t\in\{m+1,\ldots,s\}$.  

\begin{figure}[h!]
\begin{tabular}{ccc}
\includegraphics[scale=0.37]{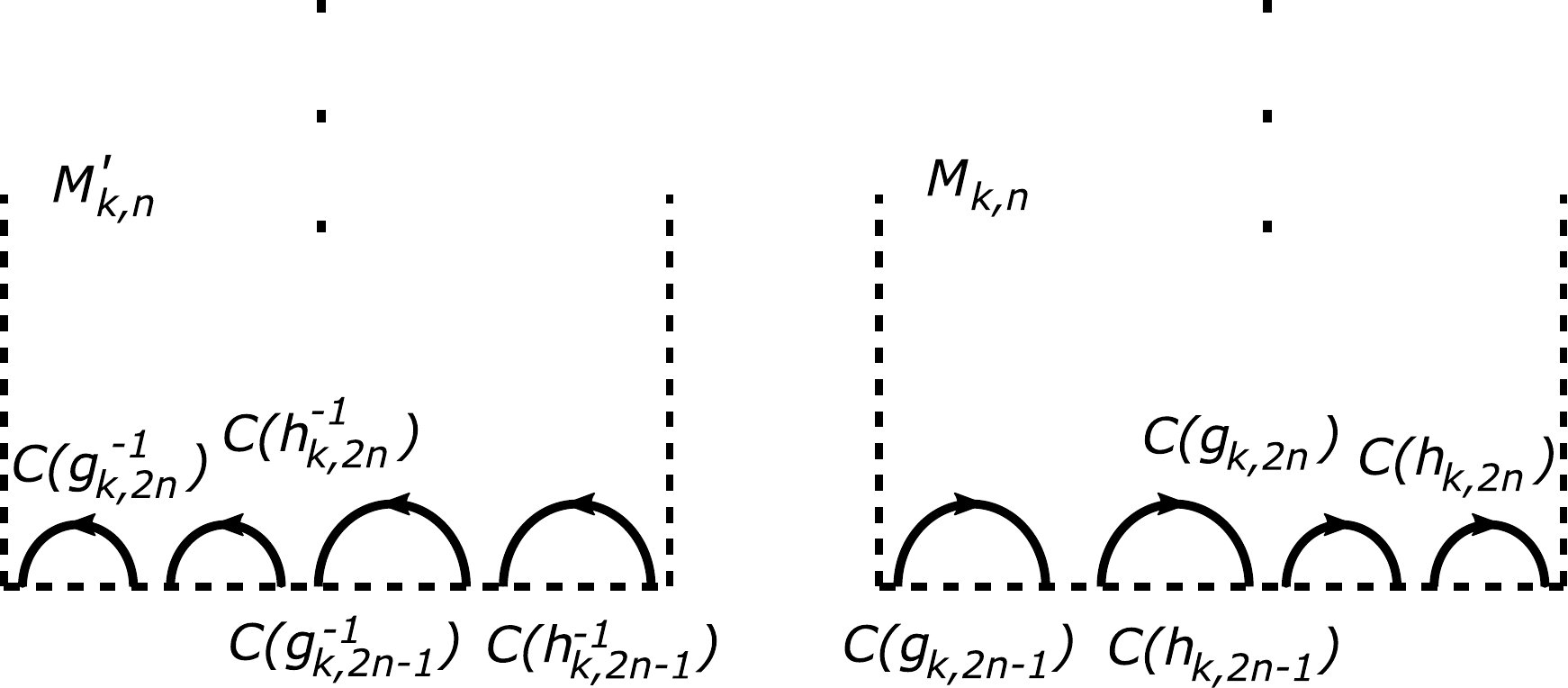}&\includegraphics[scale=0.37]{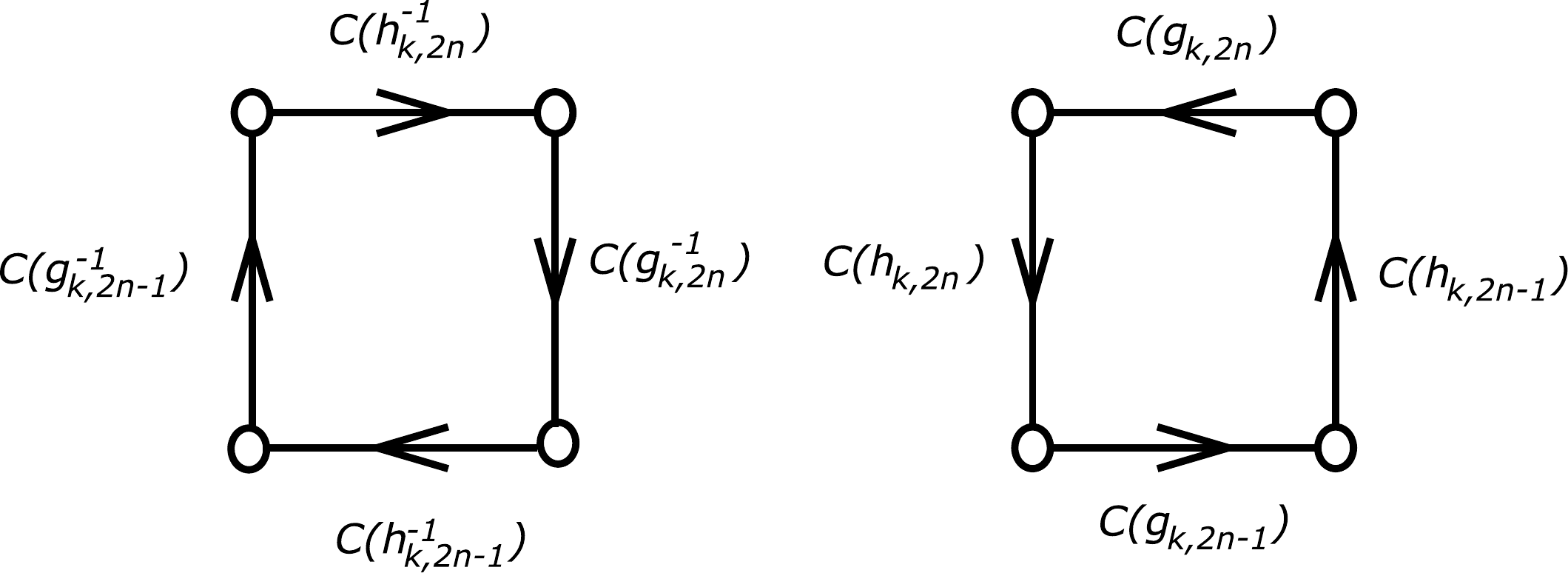}\\
a. The image of $(M_{k,n}\cup \hat{M}_{k,n})\setminus K_1$ under $\pi$ & b. The subsurface $S_k$\\
is the subsurface $S_k$. &is topologically equivalent to\\ & two squares with identifications. \\
&&\\
\includegraphics[scale=0.37]{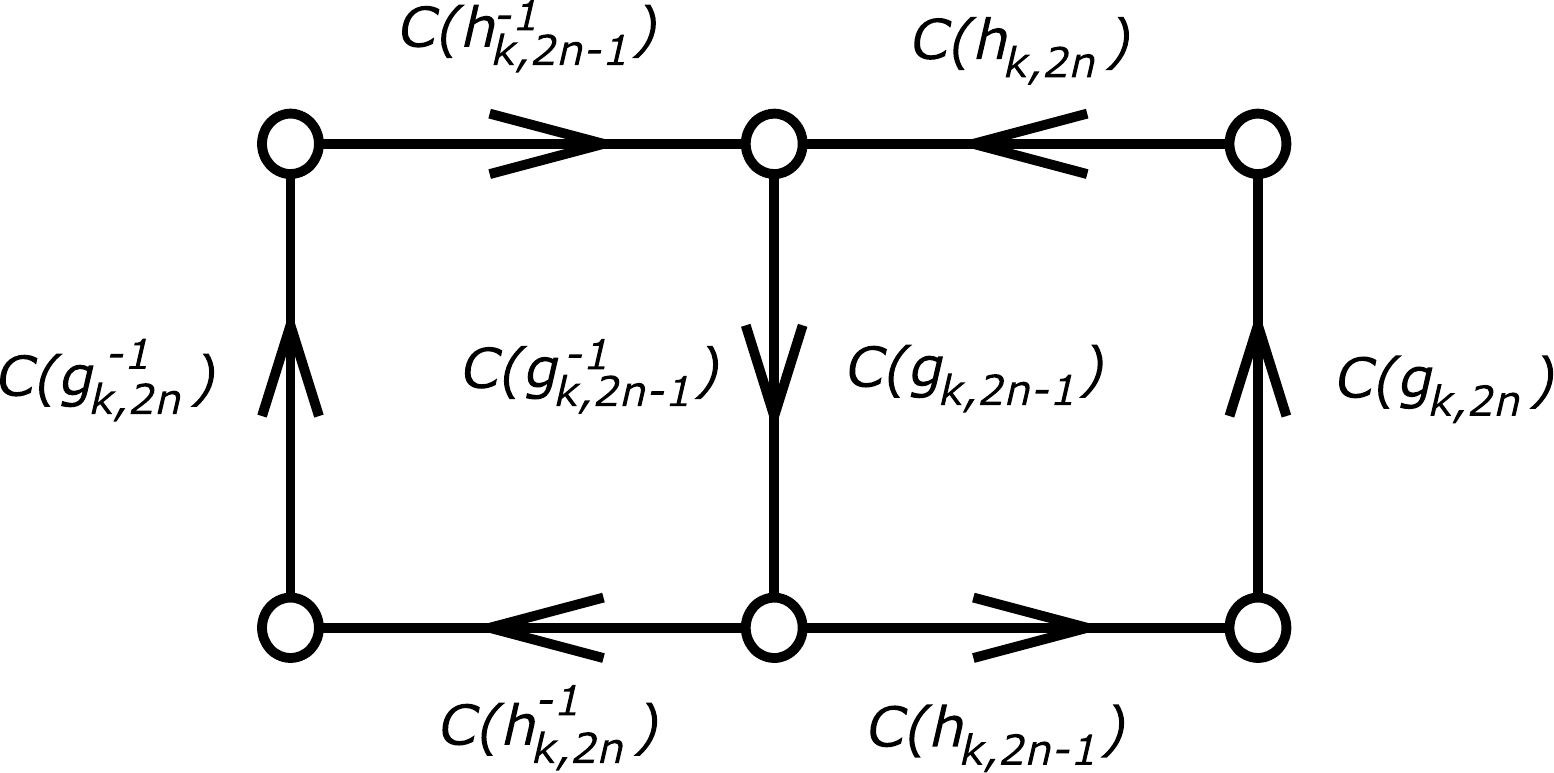}&\includegraphics[scale=0.37]{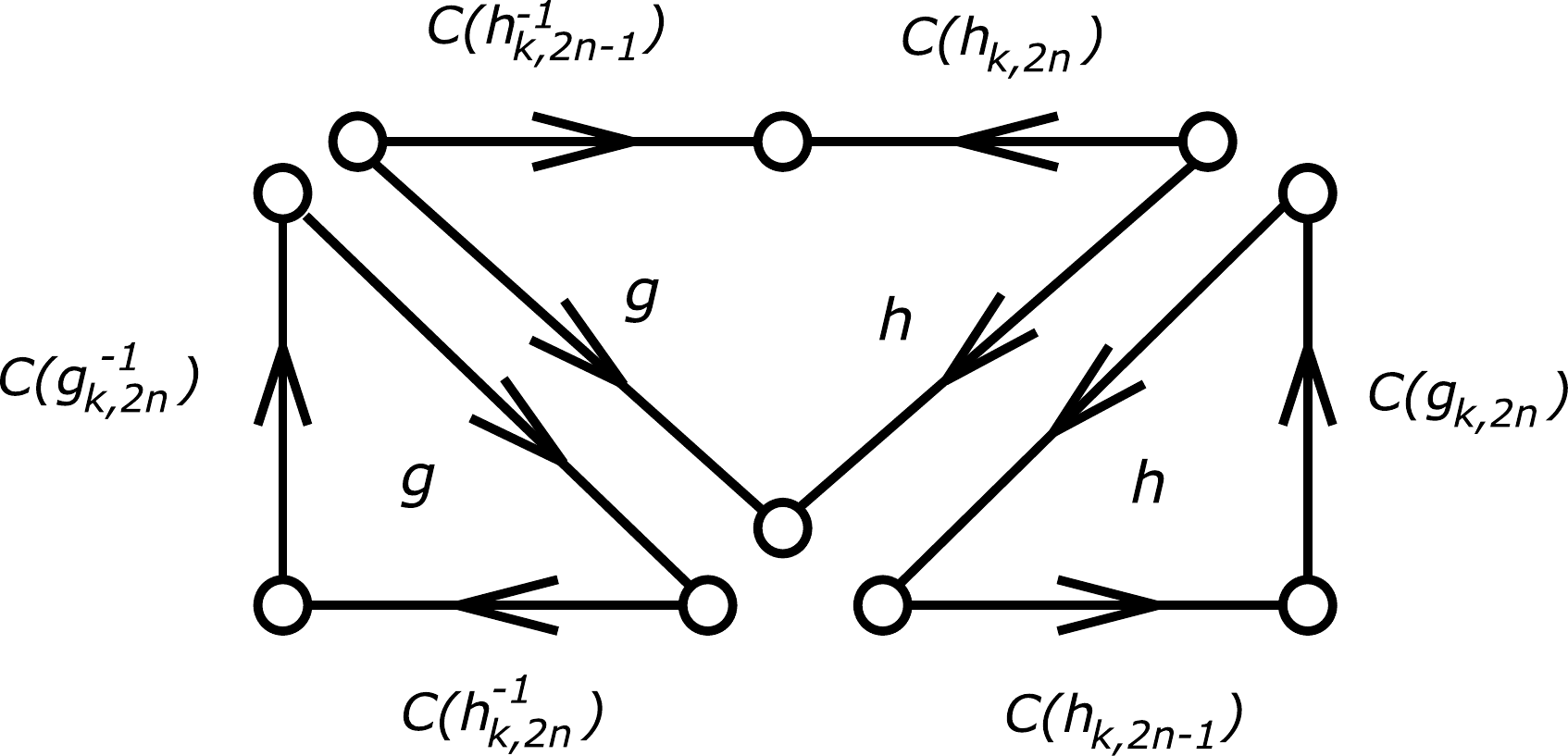}\\
c. Identifying the sides $C(g_{k, 2n}^{-1})$ and $C(h_{k, 2n})$. & d. Cutting we get three triangles.\\
&&\\
\includegraphics[scale=0.37]{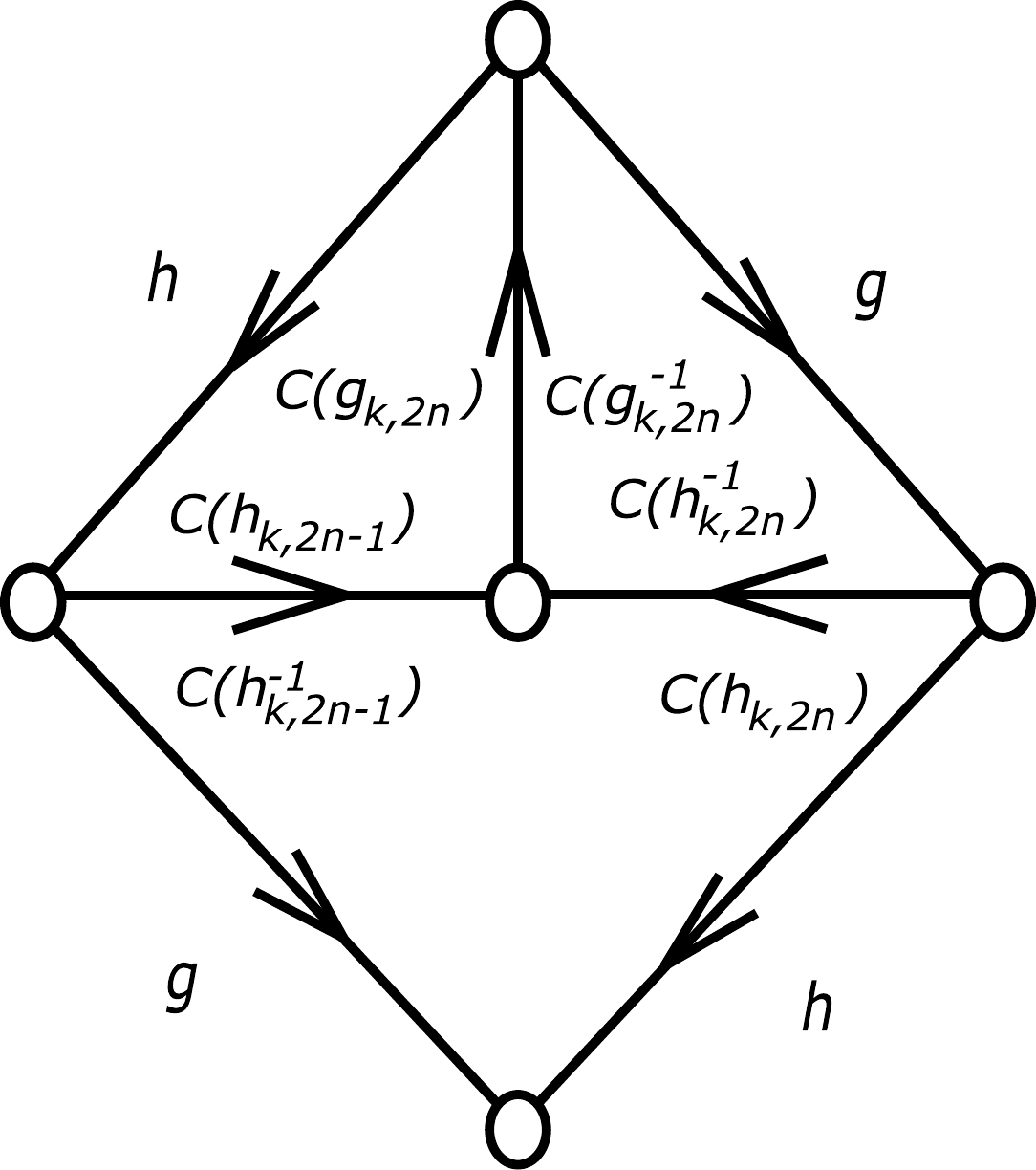}&\includegraphics[scale=0.37]{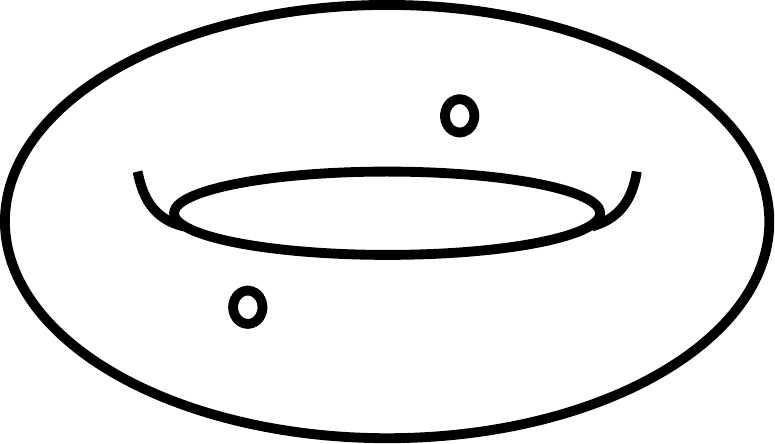}\\
e. Identifying the sides $C(g_{k,2n})$ and $C(g^{-1}_{k,2n})$; & f. Finally is obtained a torus \\
$(C(h_{k,2n-1}))$ and $(C(h^{-1}_{k,2n-1}))$;& with two punctured.\\
$(C(h_{k,2n}))$ and $(C(h^{-1}_{k,2n}))$, respectively. &\\
\end{tabular}
\caption{\emph{Graphic description of the process of taking a fundamental region on the Hyperbolic plane and through the quotient with a geometric Schottky group arise a non-compact Riemann surface.}}
   \label{toro}
\end{figure}

Now, we shall prove that  the end $[C_{k,l}]_{l\in\mathbb{N}}$ with $k\in\{1,\ldots,m\}$ have infinite genus. In other words, we will prove that each subsurface $C_{k,l}$ of $S_{m,s}$ has infinite genus with $k\in\{1,\ldots,m\}$ and $l\in\mathbb{N}$.

Given the end $[C_{k,l}]_{l\in\mathbb{N}}$ and the closed intervals $I_k=\bigcup\limits_{k\in\mathbb{N}}I_{k,n}$ and $\hat{I}_k=\bigcup\limits_{k\in\mathbb{N}}\hat{I}_{k,n}$ (see equations (\ref{eq:14}) and (\ref{eq:15})) with $k\in\{1,\ldots, m\}$, we define the two vertical strips
\[
\begin{array}{ccl}
M_{k,n} & := &\{z\in\mathbb{H}: Re(z)\in Int(I_{k, 2n-1}\cup I_{2k,n})  \text{ and, } 0< Im(z) < 1\}\subset F(\Gamma_{m,s}), \\
\hat{M}_{k,n}& := &\{z\in\mathbb{H}: Re(z)\in Int(\hat{I}_{k, 2n-1}\cup \hat{I}_{2k,n}) \text{ and, } 0< Im(z)<1\}\subset F(\Gamma_{m,s}).
\end{array}
\]

For $l=1$, the image $(M_{k,n}\cup \hat{M}_{k,n}) \setminus {K_1}$ under $\pi$ is a subsurface of $C_{t,1}$, denote as $S_k:=\pi((M_{k,n}\cup \hat{M}_{k,n})\setminus {K_1})$, which is homeomorphic to the torus punctured by $2$ points, for each $k\in\mathbb{N}$ (see Figure \ref{toro}). We remark that if $u\neq k\in\mathbb{N}$, then the subsurface $S_u$ and $S_k$ are disjoint, $S_u\cap S_k=\emptyset$, this implies that the subsurface $C_{t,1}$ of $S_{m,s}$ has infinite genus.


For $l=2$, the image of $(M_{k,n}\cup \hat{M}_{k,n}) \setminus {K_2}$ under $\pi$ is a subsurface of $C_{t,2}$, denote as $S_k:=\pi((M_{k,n}\cup \hat{M}_{k,n})\setminus {K_2})$, which is homeomorphic to the torus punctured by $2$ points, for each $k>1$. We remark that if $u\neq k\in\mathbb{N}$, then the subsurface $S_u$ and $S_k$ are disjoint, $S_u\cap S_k=\emptyset$, this implies that the subsurface $C_{t,2}$ of $S_{m,s}$ has infinite genus.

For each $l\in\mathbb{N}$, the image $(M_{k,n}\cup \hat{M}_{k,n})\setminus {K_l}$ under $\pi$ is a subsurface of $C_{t,l}$, denote as $S_k:=\pi((M_{k,n}\cup \hat{M}_{k,n})\setminus {K_l})$, which is homeomorphic to the torus punctured by $2$ points, for each $k>l$. We remark that if $u\neq k\in\mathbb{N}$, then the subsurface $S_u$ and $S_k$ are disjoint, $S_u\cap S_k=\emptyset$, this implies that the subsurface $C_{t,l}$ of $S_{m,s}$ has infinite genus. Hence we conclude that the end $[C_{k,l}]_{l\in\mathbb{N}}$ has infinite genus, for each $k\in\{1,\ldots,m\}$.

\qed

\begin{coro}
The fundamental group of the Riemann surface $S_{s,m}$ is isomorphic to $\Gamma_{s,m}$.
\end{coro}

\begin{coro}
Let $\Gamma_{s}$ be the subgroup of $PSL(2,\mathbb{R})$ generated by the set $\{f_1,f_1^{-1},\ldots, f_{s-1},f_{s-1}^{-1}\}$ (see equation \ref{eq:12}) then the quotient $S_s:=\mathbb{H}/ \Gamma_{s}$ is complete Riemann surface having exactly $s$ ends and genus zero. Moreover, the fundamental group of the Riemann surface $S_{s}$ is isomorphic to $\Gamma_{s}$.
\end{coro}

\begin{obs}
Following the same ideas of the proof of Theorem \ref{T:1} it is easy to check that the Rieman surface $\mathbb{H}/\Gamma$ has ends space homeomorphic to the closure of the set $\{\frac{1}{n}:n \in \mathbb{N}\}\subset \mathbb{R}$, where $\Gamma$ is the set generated by $\{f_t, f_t^{-1}:t\in\mathbb{N}\}$ (see equation \ref{eq:12}).
\end{obs}
\section*{Acknowledgements}
The second author was partially supported by UNIVERSIDAD NACIONAL DE COLOMBIA, SEDE MANIZALES. Camilo Ram\'irez Maluendas have dedicated this work to his beautiful family: Marbella and Emilio, in appreciation of their love and support.




\begin{bibdiv}
 \begin{biblist}

\bib{Abi}{article}{
   author={Abikoff, William},
   title={The uniformization theorem},
   journal={Amer. Math. Monthly},
   volume={88},
   date={1981},
   number={8},
   pages={574-592},
   issn={0002-9890},
}


\bib{AyC}{article}{
   author={Arredondo, John A.},
   author={Ram\'\i rez Maluendas, Camilo},
   title={On the infinite Loch Ness monster},
   journal={Comment. Math. Univ. Carolin.},
   volume={58},
   date={2017},
   number={4},
   pages={465--479},
}

\bib{AVR}{article}{
   author={Arredondo, John A.},
   author={Maluendas, Camilo Ram\'\i rez},
   author={Valdez, Ferr\'an},
   title={On the topology of infinite regular and chiral maps},
   journal={Discrete Math.},
   volume={340},
   date={2017},
   number={6},
   pages={1180--1186},
}

\bib{Bear1}{book}{
   author={Beardon, Alan F.},
   title={A Premier on Riemann Surfaces},
   series={London Mathematical Society Lecture Note Series},
   volume={78},
   publisher={Cambridge University Press, Cambridge},
   date={1984},
   pages={x+188},
   isbn={0-521-27104-5},
}

\bib{Bear}{book}{
   author={Beardon, Alan F.},
   title={The geometry of discrete groups},
   series={Graduate Texts in Mathematics},
   volume={91},
   publisher={Springer-Verlag, New York},
   date={1983},
   pages={xii+337},
   isbn={0-387-90788-2},
}



\bib{BJ}{article}{
   author={Button, Jack},
   title={All Fuchsian Schottky groups are classical Schottky groups},
   conference={
      title={The Epstein birthday schrift},
   },
   book={
      series={Geom. Topol. Monogr.},
      volume={1},
      publisher={Geom. Topol. Publ., Coventry},
   },
   date={1998},
   pages={117--125 (electronic)},
}

\bib{Carne}{book}{
       author={Carne, T. K},
       title={Geometry and groups},
        publisher={Cambridge University (electronic)},
        date={2012},
}



\bib{Far}{book}{
   author={Farkas, H. M.},
   author={Kra, I}
   title={Riemann Surfaces},
   series={Graduate Texts in Mathematics},
   volume={71},
   edition={2},
   publisher={Springer-Verlag, New York},
   date={1992},
   pages={xvi+363},
   isbn={0-387-97703-1},
   doi={10.1007/978-1-4612-2034-3},
}
	
\bib{Ford}{article}{
   author={Ford, L. R.},
   title={The fundamental region for a Fuchsian group},
   journal={Bull. Amer. Math. Soc.},
   volume={31},
   date={1925},
   number={9-10},
   pages={531--539},
}



\bib{Fre}{article}{
   author={Freudenthal, Hans},
   title={\"Uber die Enden topologischer R\"aume und Gruppen},
   language={German},
   journal={Math. Z.},
   volume={33},
   date={1931},
   number={1},
   pages={692--713},
}



\bib{Ghys}{article}{
   author={Ghys, {\'E}tienne},
   title={Topologie des feuilles g\'en\'eriques},
   language={French},
   journal={Ann. of Math. (2)},
   volume={141},
   date={1995},
   number={2},
   pages={387--422},
}
	







\bib{KS}{book}{
   author={Katok, Svetlana},
   title={Fuchsian groups},
   series={Chicago Lectures in Mathematics},
   publisher={University of Chicago Press, Chicago, IL},
   date={1992},
   pages={x+175},
}



\bib{KS2}{article}{
   author={Katok, Svetlana},
   title={Fuchsian groups, geodesic flows on surfaces of constant negative
   curvature and symbolic coding of geodesics},
   conference={
      title={Homogeneous flows, moduli spaces and arithmetic},
   },
   book={
      series={Clay Math. Proc.},
      volume={10},
      publisher={Amer. Math. Soc., Providence, RI},
   },
   date={2010},
   pages={243--320},
}

\bib{Ker}{book}{
   author={Ker\'ekj\'art\'o, B\'ela.},
   title={Vorlesungen \"uber Topologie I},
   series={Mathematics: Theory \& Applications},
   publisher={Springer},
   place={Berl\'in},
   date={1923},
   }


\bib{LJ}{book}{
   author={Lee, John M.},
   title={Introduction to topological manifolds},
   series={Graduate Texts in Mathematics},
   volume={202},
   publisher={Springer-Verlag, New York},
   date={2000},
   pages={xviii+385},
}


\bib{LKTR}{article}{
   author={Lindsey, Kathryn},
   author={Trevi\~no, Rodrigo},
   title={Infinite type flat surface models of ergodic systems},
   journal={Discrete Contin. Dyn. Syst.},
   volume={36},
   date={2016},
   number={10},
   pages={5509--5553},
}

\bib{TM}{article}{
   author={Maitani, Fumio},
   author={Taniguchi, Masahiko},
   title={A condition for a circle domain and an infinitely generated
   classical Schottky group},
   conference={
      title={Topics in finite or infinite dimensional complex analysis},
   },
   book={
      publisher={Tohoku University Press, Sendai},
   },
   date={2013},
   pages={169--175},
}


\bib{MB}{book}{
   author={Maskit, Bernard},
   title={Kleinian groups},
   series={Grundlehren der Mathematischen Wissenschaften [Fundamental
   Principles of Mathematical Sciences]},
   volume={287},
   publisher={Springer-Verlag, Berlin},
   date={1988},
   pages={xiv+326},
}

\bib{MatKat}{article}{
   author={Matsuzaki, Katsuhiko},
   title={Dynamics of Teichm\"uller modular groups and topology of moduli
   spaces of Riemann surfaces of infinite type},
   journal={Groups Geom. Dyn.},
   volume={12},
   date={2018},
   number={1},
   pages={1--64},
}

\bib{PSul}{article}{
   author={Phillips, Anthony},
   author={Sullivan, Dennis},
   title={Geometry of leaves},
   journal={Topology},
   volume={20},
   date={1981},
   number={2},
   pages={209--218},
}






\bib{Ian}{article}{
   author={Richards, Ian},
   title={On the classification of noncompact surfaces},
   journal={Trans. Amer. Math. Soc.},
   volume={106},
   date={1963},
   pages={259--269},
}





\bib{SPE}{article}{
   author={Specker, Ernst},
   title={Die erste Cohomologiegruppe von \"Uberlagerungen und
   Homotopie-Eigenschaften dreidimensionaler Mannigfaltigkeiten},
   language={German},
   journal={Comment. Math. Helv.},
   volume={23},
   date={1949},
   pages={303--333},
}

\bib{Sch}{book}{
   author={Schwartz, Richard},
   title={Mostly Surfaces},
   publisher={AMS, Volume 60},
   date={2011},
}

\bib{ValRa}{article}{
   author={Ram\'\i rez Maluendas, Camilo},
   author={Valdez, Ferr\'an},
   title={Veech groups of infinite-genus surfaces},
   journal={Algebr. Geom. Topol.},
   volume={17},
   date={2017},
   number={1},
   pages={529--560},
}


\bib{Ziel}{book}{
    author={Zielicz, Anna M.},
    title={Geometry and dynamics of infinitely generated Kleinian groups-Geometrics Schottky groups},
    series={PhD Dissertation},
    publisher={Universit\"{a}t Bremen},
    date={2015},
}









\end{biblist}
\end{bibdiv}
\end{document}